\documentclass[reqno]{amsart}
\usepackage{graphicx}
\usepackage{amsthm,amsmath,verbatim,amssymb}
\usepackage{subcaption}
\usepackage{caption}
\usepackage{color}
\usepackage{arydshln}
\usepackage[toc,page]{appendix}

\newtheorem{prop}{Proposition}
\newtheorem{thm}{Theorem}

\newcommand{\abs}[1]{\left\vert#1\right\vert}

\newcommand{\C}{\mathbb{C}}
\newcommand{\E}{\mathbb{E}}
\newcommand{\N}{\mathbb{N}}
\newcommand{\PR}{\mathbb{P}}

\newcommand{\D}{\mathbb{D}}
\title[Zeros of repeated derivatives]{Zeros of repeated derivatives of random polynomials}
\author[Feng]{Renjie Feng}
\author[Yao]{Dong Yao}
\address{Beijing International Center for Mathematical Research, Peking University,  China}
\email{renjie@math.pku.edu.cn}
\address{Duke University, USA}
\email{dong.yao@duke.edu}

\begin{document}

\date{\today}
   \maketitle
   \begin{abstract}
It has been shown that zeros of Kac polynomials $K_n(z)$ of degree $n$ cluster asymptotically near the unit circle as $n\to\infty$ under some assumptions. This property remains unchanged for the $l$-th derivative of the Kac polynomials $K^{(l)}_n(z)$ for any fixed order $l$.  So it's natural to study the situation when the number of the derivatives we take depends on $n$, i.e., $l=N_n$.  We will show that the limiting global behavior of zeros of $K_n^{(N_n)}(z)$ depends on the limit of the ratio $N_n/n$. In particular,  we prove that when the limit of the ratio is strictly positive, the property of the uniform clustering around the unit circle  fails; when the ratio is close to 1, the zeros have some rescaling phenomenon. Then we study such problem for random polynomials with more general coefficients. But things, especially the rescaling phenomenon, become very complicated for the general case when $N_n/n\to 1$, where we compute the case of the random elliptic polynomials to illustrate this. 

  \end{abstract}
\section{Introduction}
There are many well known results regarding the nontrivial relations between zeros and critical points of polynomials. The classical Gauss-Lucas theorem states that all the critical points of a polynomial are in the convex hull of its zeros, in particular, if  all the zeros are real, then so are the
zeros of the derivative.
 Differentiating a polynomial which has only real zeros will  even out zero spacings  \cite{DR}; in the case of  random
trigonometric polynomials, it's proved in \cite{DY} that the repeated differentiation causes the roots of the function to approach equal spacing, which can be viewed as a toy model of crystallization in one dimension. For random polynomials under some mild assumptions,   the distribution of critical points and the distribution of its zeros are asymptotically the same as the degree tends to infinity. This is because, roughly speaking,  the coefficients of the derivative of a random polynomial are not changed dramatically. Actually, such result  holds for  any fixed number of differentiation \cite{KZ}. 
In this article, we are primarily interested in the case  when the number of the derivatives  we take for the random polynomials is not fixed but grows to infinity with the degree.

Our starting point is the classical Kac polynomials. Let $\xi_0, \xi_1,\cdots$  be nondegenerate, independent and
identically distributed (i.i.d.) complex random variables. The Kac polynomials are defined as
\begin{equation}\label{kacdef}
	K_n(z) =\sum_{k=0}^n \xi_k z^k. 
\end{equation}
The Kac polynomials have degree $n$ almost surely by assuming
\begin{equation}\label{condition2}
\PR(\xi_0=0)=0.
\end{equation}
The distribution of zeros of Kac polynomials has been  studied for decades,  we refer to \cite{HKPV, IZ1,IZ,  Kac, KZ2, KZ, Sodin, SV} and the references therein. It's proved that if
\begin{equation}\label{condition1}
\E\log(1+|\xi_0|)< \infty,
\end{equation}
then with probability $1$, the empirical measure of zeros of Kac polynomials 
converges weakly to the uniform probability measure on the unit circle as  $n$  tends to infinity  \cite{IZ1, IZ, KZ2, KZ, SV}.  If the assumption \eqref{condition1} is removed, then zeros of $K_n(z)$ may not concentrate around the unit circle, see \cite{IZ, KZ2} for the case when $\abs{\xi_0}$ has some logarithmic tails.

The property of clustering around the unit circle remains unchanged for the $l$-th derivative of the Kac polynomials $K^{(l)}_n(z)$ for any fixed $l$ as $n$ tends to infinity \cite{KZ}. But things become interesting if the number of the derivatives we take depends on $n$, e.g., $l=N_n$. For the extreme case when $N_n=n$, there is no zero for $K_n^{(n)}$ almost surely. Hence, some natural questions are:
\begin{itemize} \item What is the critical growth order of $N_n$ so that the  property of clustering around the unit circle for the Kac polynomials $K_n^{(N_n)}$ fails?\item When it fails, what is the distribution of zeros of  $K_n^{(N_n)}$?  And how does the distribution depend on the growth order of $ N_n$? \end{itemize}  In this article, we will answer these questions for the Kac polynomials completely. The estimates we derive for the Kac case can be applied to  the general random polynomials. But there are some issues for the general random polynomials, where we will compute the case of the random elliptic polynomials to illustrate this.  

\subsection{Notations}
Before we state our main results, we need to introduce some notations. 
We denote by 
 \begin{equation}\label{general}
   p_n(z)=\sum_{k=0}^n \xi_k p_{k,n} z^k,
 \end{equation} the random polynomials of degree $n$ with general coefficients,
  where $p_{k,n}$ are deterministic coefficients and  $\xi_k$ are nondegenerate i.i.d. complex random variables.  Throughout the article, we assume the random variable $\xi_0$ satisfies the conditions  \eqref{condition2} and \eqref{condition1}.

 We denote by $p_n^{(N_n)}(z)$ the $N_n$-th derivative of $p_n(z)$ with the degree \begin{equation}\label{dn}D_n=n-N_n.\end{equation}
 Without loss of generality, we may assume the convergence of  \begin{equation}\frac{N_n}{n}\to a \in [0,1].\end{equation}
The  random measure of zeros of  $p_n(z)$ is denoted by \begin{equation}\mu_n= \sum_{z:\,p_n(z)=0}  \delta_z,\end{equation} 
and  we use the notation  \begin{equation}\label{um}\mu_{D_n}= \sum_{z :\,p^{(N_n)}_n(z)=0}  \delta_z\end{equation}
for the  random measure of zeros of $p^{(N_n)}_n(z)$ of degree $D_n$.

Similarly, we denote by $\mu_n^K$ and $\mu_{D_n}^K$ the random  measures of zeros of $K_n(z)$ and $K^{(N_n)}_n(z)$  for the Kac polynomials, respectively, and we denote by $\mu_n^E$ and $\mu_{D_n}^E$ the random elliptic polynomials. We denote by $\D_r$ the open disk of radius $r$ centered at the origin in the complex plane.
The convergence of the random measures $\nu_n$  to $\nu$ in probability (or in distribution) means the convergence in probability (or in distribution) in the weak sense, i.e., $\int_X \phi \nu_n(dx)\to \int_X \phi \nu(dx)$ in probability (or in distribution) for any smooth test function $\phi$ with compact support. Given a measure $\nu$   on the complex plane, we define the scaling operator $$(\mathcal S_h\nu)(B)=\nu(\frac B h)$$ for $h>0$ where $B$ is any Borel set in $\mathbb{C}$.  In the end, we set $a\wedge b=\min\{a,b\}$ and $a\vee b=\max\{a,b\}$ and set $\log 0=-\infty$.

\subsection{Kabluchko-Zaporozhets Theorem}
There are many well known results regarding the global distribution of zeros of some special Gaussian random analytic functions where the ensembles are usually invariant under some group action, such as the Gaussian elliptic polynomials and Gaussian hyperbolic analytic functions \cite{HKPV, Sodin}. Recently, a remarkable result proved  in \cite{KZ}  deals with more general random analytic functions. In \cite{KZ}, Kabluchko-Zaporozhets proved that under certain assumptions on the coefficients of the random analytic functions, the distribution of zeros will converge to a deterministic rotationally invariant measure on a domain of the complex plane. Such measure can  be explicitly characterized in terms of the coefficients.  To be more precise, let's consider the random analytic function in the form of \begin{equation}\label{analytic}F_n(z)=\sum_{k=0}^{\infty} \xi_kp_{k,n}z^k,\end{equation}where $\xi_k$ are nondegenerate i.i.d. complex random variables satisfying condition \eqref{condition1} and the coefficients $p_{k,n}$ satisfy the following assumptions,
\\
 
\textbf{A1}: Assume there is a function $p:[0,\infty)\to [0,\infty)$ and a number $T_0 \in (0,\infty]$ such that

\begin{enumerate}
  \item[ 1.] $p(t)>0$ for $t<T_0$ and $p(t)=0$ for $t>T_0$.
  \item[ 2.] $p$ is continuous on $[0,T_0)$, and in the case $T_0<\infty$, left continuous at $T_0$.
  \item[ 3.] $\lim\limits_{n\to\infty}\sup_{k\in[0,An]} \abs{\abs{p_{k,n}}^{\frac{1}{n}}-p(\frac{k}{n})  }=0$
       for every $A>0$.
  \item[ 4.] $R_0=\liminf\limits_{t\to\infty}p(t)^{-1/t}\in (0,\infty]$, $\liminf\limits_{k\to\infty}|p_{k,n}|^{-1/k}\geq R_0$ for every fixed $n \in \N$ and additionally, $\liminf\limits_{n,k/n \to \infty} \abs{p_{k,n}}^{-1/k}\geq R_0$.
\end{enumerate}
 
 Roughly speaking, the major assumption is that the coefficients $p_{k,n}$ are approximately $e^{n\log p(\frac{k}{n})}$ for some $p$, which is positive on some interval $[0,T_0)$, continuous in $[0,T_0]$ and   equal to $0$ in $(T_0,\infty)$.  We have, 
\begin{thm}(Kabluchko-Zaporozhets  \cite{KZ})\label{theorem1}
Under \textbf{A1} and \eqref{condition1}, let $I(s)$ be the Legendre-Fenchel transform of $-\log p$,  i.e., $I(s)=\sup_{t\geq 0}(st+\log p(t))$,  then the random measure $\frac{1}{n}\mu_{F_n}$ of zeros of $F_n(z)$  converges in probability to a deterministic measure $\mu$ in $\D_{R_0}$, which is rotationally invariant and satisfies
\begin{equation*}
  \mu(\D_r)=I'(\log r), \quad\quad r\in(0,R_0).
\end{equation*}
\end{thm}
As a convention, $I'$ is the left derivative of $I$. A typical example to apply Kabluchko-Zaporozhets Theorem is the Kac polynomials where we have \begin{equation}\label{examplekac1}p_{k,n}=1_{k\leq n},\,\, p(t)=1_{t\leq 1},\,\, T_0=1.\end{equation} By some computations, we have $I(s)=s\vee 0$ and thus the limiting distribution satisfies
\begin{equation}\label{examplekac}
\mu(\D_r)=\begin{cases}
0 &0\leq r\leq 1,\\ 1 &r>1,
\end{cases}
\end{equation}
 i.e., the uniform probability measure on the unit circle.

 But we can not apply Kabluchko-Zaporozhets Theorem directly in our case to derive the distribution of zeros of $K_n^{(N_n)}$ or that of the general random polynomials $p_n^{(N_n)}$. For example, if $N_n=n-\left\lfloor \log n \right\rfloor$, then the degree of $p_n^{(N_n)}$ is $D_n=\left\lfloor \log n \right\rfloor$, therefore,  one can not find some $A$ so that  the assumption 3. in \textbf{A1} is satisfied.   We need to modify their theorem 
to deal with our situation more conveniently.  We consider the random polynomials in the form of 
\begin{equation}\label{analytic22}F_n(z)=\sum_{k=0}^{ (T_0-\delta_n)L_n} \xi_kp_{k,n}z^k, \end{equation}
where  $(T_0-\delta_n)L_n$ is an integer and we assume that $F_n(z)$ satisfies the following assumptions,  
\\

\textbf{A2}: There exists a function $p:[0,\infty)\to [0,\infty)$, a positive number $T_0 \in (0,\infty)$, a sequence of positive integers $L_n$ going to $\infty$  as $n\to\infty$, a sequence of numbers $\delta_n \in(-T_0,T_0)$ (not necessarily positive) that goes to 0 as $n \to \infty$ such that
\begin{enumerate}
  \item [1.] $p(t)>0$ for $t\in [0,T_0)$ and $p(t)=0$ for $t> T_0$.
  \item [2.]$p$ is continuous in $[0,T_0]$.
  \item [3.]$\lim\limits_{n\to\infty}\sup_{0\leq k \leq (T_0-\delta_n)L_n}\abs{\abs{p_{k,n}}^{\frac{1}{L_n}}-p(\frac{k}{L_n}\wedge T_0 )}=0.$
\end{enumerate} 

 Then we have the following theorem where the proof is sketched in Appendix A, 
\begin{thm}\label{theorem2}
For random polynomials  $F_n(z)$ in the form of \eqref{analytic22} which satisfy the assumptions \textbf{A2}, let I(s) be the Legendre-Fenchel transform of $-\log p$, then the random measure   $\frac{1}{L_n}\mu_{F_n}$ of zeros will converge in probability to a deterministic rotationally invariant measure $\mu$ where 
\begin{equation}\label{limiting measure}
  \mu(\D_r)=I'(\log  r),\,\,\,\,\,\, \,\,\,r>0.
\end{equation}
\end{thm}

Throughout the article, we often make use of the following estimate 
\begin{equation}\label{assumption 2 point 3}
  \lim\limits_{n\to\infty} \sup_{0\leq k \leq (T_0-\delta_n)L_n} \abs{\frac{1}{L_n}\log |p_{k,n}|-\log p(\frac{k}{L_n}\wedge T_0) }=0.
\end{equation}
This estimate implies the main assumption 3. in  \textbf{A2}, which is the direct consequence of the following inequality
$$\label{equationddd}
  \abs{x-y}\leq (x\wedge y)e^{\abs{\log x-\log y}}\abs{\log x-\log y}
$$for any $x,y>0$. 
The main advantage of \eqref{assumption 2 point 3} is the convenience in computations.

\subsection{Main results}
We first state our main results for the Kac polynomials, which will answer the questions we raised at the beginning of the article.   
\subsubsection{Kac polynomials}
 The main result is that the limiting behavior of the distribution of zeros of  $K_n^{(N_n)}$ will depend on the limit of the ratio $N_n/n$.  We will divide our discussions into two categories: $D_n$ goes to infinity and $D_n$ remains to be a fixed number, where $D_n=n-N_n$ is the degree of the random polynomials $K_n^{(N_n)}$. 
Without loss of generality, we consider the following four different cases 
\textcircled{1}$N_n/n\to 0$; \textcircled{2}$N_n/n\to a\in (0,1)$; \textcircled{3} $N_n/n\to 1$ and $D_n\to\infty$, e.g., $N_n= n-\left\lfloor \log n \right\rfloor$ and $D_n=\left\lfloor \log n \right\rfloor$; \textcircled{4} $N_n/n\to 1$ but  $D_n=m<\infty$, i.e., $K_n^{(N_n)}$ has a fixed degree $m$.  

In the cases of \textcircled{1}\textcircled{2}\textcircled{3} where $D_n\to\infty$, we will show that the coefficients of $K_n^{(N_n)}$ or its rescaling will satisfy the assumptions \textbf{A2} with different choices of $L_n$, $\delta_n$, $T_0$ and $p$, then we apply Theorem \ref{theorem2} to prove
\begin{thm}\label{theorem3}
 Assume $D_n\to\infty$ as $n\to\infty$,  we have the following results regarding the empirical measure of zeros of derivatives of Kac polynomials $K^{(N_n)}_n$, 
\begin{enumerate}
  \item If $\lim\limits_{n\to\infty}\frac{N_n}{n}=0 $, then $\frac{1}{D_n}\mu^K_{D_n} $ converges in probability to the uniform probability measure on the unit circle, i.e., the measure defined in \eqref{examplekac};

  \item If $\lim\limits_{n\to \infty}\frac{N_n}{n}=a \in (0,1)$, then $\frac{1}{D_n}\mu^K_{D_n}$ converges in probability to a rotationally invariant measure $\mu_a^K$ on $\C$ defined by
      \begin{equation}\label{ddddd}\mu^K_a(\D_r)=\begin{cases}
\frac{ar}{(1-a)(1-r)} &0<r<1-a, \\ 1 & r\geq 1-a;
\end{cases}
\end{equation}

  \item If $\lim\limits_{n\to \infty}\frac{N_n}{n}=1$, then globally we have the following  convergence in probability\begin{equation}\label{fdff}\frac{1}{D_n}\mu^K_{D_n}\to \delta_0 .   \end{equation} 
   If we set  $R_n=\frac{n}{D_n}$ as the quotient of the degrees of $K_n$ and $K_n^{(N_n)}$ and consider the rescaling Kac polynomials $\tilde K_{n}(z):=K_{n}^{(N_n)}(\frac{z}{R_n})$, then the empirical measure $\frac{1}{D_n}\mu_{D_n}^{\tilde K}$ which is the same as $\frac{1}{D_n}\mathcal S_{R_n}(\mu_{D_n}^{ K})$ converges in probability to a rotationally invariant measure $\tilde\mu^K$ where
      \begin{equation}\label{fff}
      	 \tilde \mu^K(\D_r)= \begin{cases}
r &r<1, \\ 1 &r\geq 1.
\end{cases}
      \end{equation} 
In particular, the density for the measure $\tilde \mu^K$ is   \begin{equation}\label{ffdddddf}\tilde d^K(z)=\frac{1}{2\pi\abs{z}}1_{\abs{z}\leq 1}.\end{equation}
\end{enumerate}
\end{thm}
 In the case  \textcircled{4} when $D_n$ remains to be a fixed number, we will show that the  measure of zeros of the rescaling polynomials $K_n^{(N_n)}(\frac z n)$   will converge to some random   measure.  The main tool to prove this result is the   Rouch\'e's theorem in complex analysis.  Our result is as follows, 
\begin{thm}\label{theorem4}
Suppose $\lim_{n\to\infty}\frac{N_n}{n}=1$ and $D_n=m$ for all $n$, then globally \begin{equation}\label{deddddffff} \frac 1m\mu_{ D_n } ^K\to \delta_0,\end{equation}where the convergence is in probability. Furthermore, we have the rescaling limit
\begin{equation}\label{ddddffff} \mathcal S_{n}(\mu_{D_n}^{K}) \to \mu_{f^K_m},\end{equation} where the convergence is   in distribution and $\mu_{f^K_m}$ is the random measure of zeros of the random polynomial
\begin{equation}\label{finallabel} f^K_m(z)=\sum_{k=0}^m\frac{\xi_k}{k!}z^k.\end{equation}
\end{thm}
 
	The relationship between the results in part (3), Theorem 3 and Theorem 4 has
 an intuitive explanation. Consider the case in
 part (3), Theorem 3. We can zoom in zeros of  $K_n^{(N_n)}(z)$ in two steps. First we zoom in the zeros of $K_n^{(N_n)}(z)$ by a factor of $n$, then by Theorem 4 (treating $D_n$ as fixed for this moment) the scaled zeros will be close to the zeros of $f^K_{D_n}(z)$. Here $f^K_{D_n}(z)$ is just the function in \eqref{finallabel} with $m$ replaced by $D_n$.
 If we then zoom out
 zeros of $f^K_{D_n}$ by a factor of $D_n$(which is the degree of the polynomial $f^K_{D_n}$) then as a whole we get something close to zooming in the zeros of $K_n^{(N_n)}(z)$ by a factor of $\frac{n}{D_n}$. Taking $n$ to infinity we should get the the limit in part (3) of Theorem 3.
 This is in accordance with the fact that the expression \eqref{fff} is also the limit of the empirical measure of zeros of $f^K_{D_n}(D_nz)$ as $m\to \infty$, as shown in Theorem 2.3 of \cite{KZ}. Note that in the zooming out process, we can also replace $\sum_{k=0}^{D_n} \frac{\xi_k}{k!}(D_nz)^k$ by $\sum_{k=0}^{\infty} \frac{\xi_k}{k!}(D_nz)^k$ since Theorem 2.1 of \cite{KZ} shows the empirical measure of $\sum_{k=0}^{\infty} \frac{\xi_k}{k!}(D_nz)^k$ restricted to unit disk also converges to the measure in \eqref{fff}.
	

As a summary, we show that the clustering property of zeros around the unit circle for the derivatives of  Kac polynomials holds if and only if $N_n/n\to 0$; the conclusion (3) in Theorem \ref{theorem3} together with  Theorem \ref{theorem4} imply that, if $N_n/n\to 1$, zeros will converge to the origin with the average decay rate  $D_n/n$ which is the quotient of the degrees of $K_n^{(N_n)}$ and $K_n$. Until now we completely answer the questions we proposed at the beginning of the article for Kac polynomials.

\subsubsection{General random polynomials}
We can extend the above results for the Kac polynomials to the general random polynomials where the coefficients satisfy the assumptions \textbf{A1} in Kabluchko-Zaporozhets Theorem. 

\begin{thm}\label{theorem5}
 Suppose the random polynomial $p_n(z)$  \eqref{general} satisfies \textbf{A1} with some function $p(t)$, then regarding the zeros of $p_n^{(N_n)}$, we have, 
\begin{enumerate}
\item If $\lim_{n\to\infty} \frac{N_n}n=0$, let $I(s)$ be the Legendre-Fenchel transform of $-\log p$, then $\frac{1}{D_n}\mu_{D_n}$ converges in probability to a rotationally invariant measure $\mu$ given by $$\mu(\D_r)=I'(\log r), \quad r>0.$$That is, $\frac{1}{D_n}\mu_{D_n}$ has the same limit as $\frac{1}{n}\mu_{n}$.

\item If $\lim_{n\to\infty} \frac{N_n}n=a\in (0,1)$, let $\log u_a=\log p(t+a)+(t+a)\log (t+a)-t\log t+(1-a)\log (1-a)$ if $0\leq t\leq 1-a$ and $-\infty$ if $t>1-a$. Let $I_a(s)$ be the Legendre-Fenchel transform of $-\log u_a$, then $\frac{1}{D_n}\mu_{D_n}$ converges in probability to a rotationally invariant measure $\mu_a$ given by $$\mu_a(\D_r)=\frac{1}{1-a}I_a'(\log r), \,\,\,\, r>0.$$




\end{enumerate}

\end{thm}
Compared with Theorem \ref{theorem3} and Theorem \ref{theorem4} for the Kac case,  things become complicated for the general random polynomials when the ratio $N_n/n$ tends to $1$. First, one can not conclude that $\frac{1}{D_n}\mu_{D_n}$ converges in probability to $\delta_0$. To see this, let's consider the following example where the coefficients of the random polynomials $p_n$ are 
\[p_{k,n}=\begin{cases}
1 &0\leq k< N_n,\\
\frac{n!(k-N_n)! }{k! D_n! } &N_n\leq k\leq n,
\end{cases}
\]
where  $$D_n=\left\lfloor \log n \right\rfloor\,\,\mbox{and}\,\, N_n=n-D_n.$$
We let $$p(t)=1_{0\leq t\leq 1}.$$ We claim that $p_{k,n}$ and $p$ satisfy the assumptions \textbf{A1}.  Indeed, when $0\leq k <N_n$, we have 
$$p^{\frac{1}{n}}_{k,n}=p(\frac{k}{n}). $$
Therefore, it remains to prove
\begin{equation*}
\lim\limits_{n\to\infty} \sup_{N_n\leq k\leq n}\abs{p^{\frac{1}{n}}_{k,n}-1}=0.
\end{equation*}
By \eqref{assumption 2 point 3}, it's enough to show
\begin{equation}\label{pkn}
\lim\limits_{n\to\infty} \sup_{N_n\leq k\leq n}\abs{\frac{1}{n}\log p^{}_{k,n}}=0.
\end{equation}
For $N_n\leq k\leq n$, we have $1\leq \frac{n!(k-N_n)!}{k!D_n!} \leq \frac{n!}{k!}$, then 
\begin{align*}
\sup_{N_n\leq k\leq n}\abs{\frac{1}{n}\log p_{k,n}} \leq \sup_{N_n\leq k\leq n}\frac{1}{n} \log \frac{n!}{k!}
\leq \frac{1}{n}\log n^{D_n}
\leq \frac{\log^2 n}{n}, 
\end{align*}
	where \eqref{pkn} follows as $n\to\infty$, which completes the proof of the claim.  But the $N_n$-th derivative of $p_n$ is
$$p_{n}^{(N_n)}=\frac{n!}{D_n!} \sum_{k=0}^{ D_n}\xi_{k+N_n} z^k, $$
which is in the form of Kac polynomials, thus the empirical measure of zeros will converge to the uniform probability measure on the circle instead of the delta function at the origin.


Secondly, even if zeros converge to $\delta_0$, one can not easily find the rescaling limit of the empirical measure of zeros if there exists one. The rescaling property should highly depend on the properties of coefficients, such as the convergent rate of $p_{n,k}$ to $p(t)$ and the monotonicity of $p_{k,n}$ for each fixed $n$. The following results regarding the elliptic polynomials provide such an example.  
\subsubsection{Random elliptic polynomials}
The random elliptic polynomials are in the form of \begin{equation}\label{elliptic}
  E_n(z)=\sum_{k=0}^n \xi_k\sqrt{n\choose k} z^k.
 \end{equation}   
If  $\xi_k$ are i.i.d. complex Gaussian random variables, then the random elliptic polynomials are also called Gaussian SU(2) polynomials. The Gaussian SU(2) polynomials  can be viewed as   meromorphic functions defined on the complex projective space $\mathbb{CP}^1\cong S^2$ and a basic fact is that the distribution of its zeros is invariant under the SU(2) action. The Gaussian SU(2) polynomial is the standard model when one tries to  generalize the random polynomials to random holomorphic sections on the complex manifolds \cite{BSZ1, HKPV}.

One can show that the coefficients of the random elliptic polynomials satisfy all assumptions in \textbf{A1} with the associated function    (see also \cite{KZ})\begin{equation}\label{pe}
 \log p^E(t)=-\frac{1}{2}t\log t-\frac{1}{2}(1-t)\log (1-t)\,\,\,\, \mbox{for}\,\,\,\, 0\leq t\leq 1.
\end{equation}

\begin{thm}\label{theorem6}
For the random elliptic polynomials $E_n(z)$ defined in \eqref{elliptic},  we have 
\begin{enumerate}
\item The conclusions   in Theorem \ref{theorem5} hold for $\frac{1}{D_n}\mu^E_{D_n}$ with $p$ replaced by $p^E$ defined in \eqref{pe}. 

\item If $\lim_{n\to\infty} \frac{N_n}n=1$, then we have the following global convergence in probability $$\frac 1{D_n}\mu_{D_n}^{E} \to \delta_0. $$ Furthermore,  if  $D_n\to \infty$, then in probability, we have  $$\frac{1}{D_n} \mathcal S_{\sqrt{R_n}}( \mu^E_{D_n}) \to \mu,$$
where  $R_n=\frac{n}{D_n}$ as before and  $\mu$ is the rotationally invariant probability measure defined as
\begin{equation}\label{equationdddd}
\mu (\D_r)=\frac{r(\sqrt{4+r^2}-r)}{2},\,\,\,\,\, r\in (0,\infty).
\end{equation} 
If $D_n=m<\infty$, then the following rescaling limit holds in distribution
 $$\mathcal S_{\sqrt{n}}( \mu^E_{D_n}) \to \mu _{f^E_m},$$  
where $\mu_{f^E_m}$ is the random measure of zeros of 
$f^E_m=\sum_{k=0}^{m}\frac{\xi_k}{k! \sqrt{(m-k)!}} z^k.$




\end{enumerate}

\end{thm}

\subsection{Further remarks}
Let's compare Theorem \ref{theorem6} with the part (3) of Theorem \ref{theorem3} and Theorem \ref{theorem4} for the case when $N_n/n\to 1$. Both the  empirical measures of zeros of derivatives tend to the point mass at the origin, but the interesting result is that they converge with different decay rate. Zeros converge to the origin with the average decay rate  $D_n/n$ for the Kac case and $\sqrt{D_n/n}$ for the elliptic case, which indicates that the assumptions \textbf{A1} is not enough to extract the complete information about the convergence of zeros of the $N_n$-th derivative of general random polynomials, i.e., the main assumption  $\lim\limits_{n\to\infty}\sup_{k\in[0,An]} \abs{\abs{p_{k,n}}^{\frac{1}{n}}-p(\frac{k}{n})  }=0$
       for every $A>0$ is not enough. It seems that we need to impose additional assumptions on the rate of the convergence of  $p_{k,n}$ to $p$ for $N_n \leq k \leq n$ and the growth order of $p_{k,n}$. As in \eqref{assumption 2 point 3}, we may alternatively consider the quantities
\begin{equation}\label{additional assm1}
\eta_n:= \sup_{N_n \leq k \leq n}\abs{\frac 1n \log \abs{p_{k,n}}^{}-\log p(\frac{k}{n}) }
 \end{equation}
 and \begin{equation}\label{additional assm2}
b_n:= \sup_{N_n \leq k \leq n} \abs{p_{k,n}}.
 \end{equation}
The asymptotic properties of  $\eta_n$ and $b_n$ may play important roles in the case when $N_n/n\to 1$. Note that $\eta_n$ is identical to $0$ for the Kac polynomials and   asymptotic to $\frac {\log D_n}{4 n}+O(\frac 1n)$   for the random elliptic polynomials.  Two  questions are raised:  what are the  asymptotic properties of $\eta_n$ and $b_n$ so that zeros of $p_n^{(N_n)}$ tend to the origin; if zeros tend to the origin, how does the decay rate depend on  $\eta_n$ and $b_n$. We postpone these two problems for  further investigation. 

The paper is organized as follows. We will prove Theorem \ref{theorem3} and Theorem \ref{theorem4} for the Kac polynomials in great details in \S\ref{Kac}. The estimates for the Kac case can be applied to prove Theorem \ref{theorem5} for the general random polynomials in \S\ref{generalpolynomials}. In the end, we will prove Theorem \ref{theorem6} for the random elliptic polynomials. In  Appendix \ref{proofoftheorem2}, we will sketch the proof of Theorem \ref{theorem2}. \\

 \section{Kac polynomials}\label{Kac}
 In this section, we will prove Theorem \ref{theorem3} and Theorem \ref{theorem4} for the Kac polynomials. 
 
Let $K_{n}^{(N_n)}$ be the $N_n$-th derivative of the Kac polynomials.  
Since we want to prove the  empirical measure of zeros converges to a deterministic limit, it suffices to prove the convergence in distribution.  
By the fact that $\xi_k$ are i.i.d., it's equivalent to consider
\begin{equation}
   K_{n}^{(N_n)}(z)=\sum_{k=0}^{D_n} \xi_k (k+1)\cdots \cdot (k+N_n)z^k.
\end{equation}
Observing that the random measure of zeros is invariant by the dilation, i.e., $\mu_{cf}=\mu_f$ for any nonzero $c$, we can alternatively consider the following normalized random polynomial so that the leading order term is $\xi_{D_n}z^{D_n}$, 
\begin{equation}\label{def of fn}
 K_{n}^{(N_n)}(z)=\sum_{k=0}^{D_n} \xi_k f_{k,n} z^k,
\end{equation}
where throughout the article, we set
\begin{equation}\label{def of fk,n}
f_{k,n}:=\frac{(k+N_n)!D_n!}{k!n!}.
\end{equation}
The Stirling's formula reads
\begin{equation}
k!=c_k\sqrt{2\pi k}(\frac{k}{e})^k,
\end{equation}
where $c_k$ is a sequence of positive numbers tending to $1$ as $k$ tends to $\infty$ and hence uniformly bounded. 
Then we have
\begin{align}\label{key}
\nonumber &\frac{1}{L_n}\log f_{k,n}\\\nonumber &=\frac{1}{L_n}[(k+N_n)\log (k+N_n)-(k+N_n)+\frac{\log (k+N_n)}{2}+D_n\log D_n-D_n+\frac{\log D_n}2 )\\
\nonumber &-(k\log k-k+\frac{\log k}{2}+n\log n-n+\frac{\log n}{2} )]+\frac{1}{L_n}(\log c_{k+N_n}+\log c_{D_n}-\log c_k-\log c_n )\\
\nonumber &=\frac{1}{L_n} [ (k+N_n)\log (k+N_n)+D_n\log D_n-n\log n-k\log k ]\\
\nonumber &+\frac{1}{2L_n} (\log (k+N_n)+\log D_n-\log n-\log k)\\\nonumber &+
\frac{1}{L_n}(\log c_{k+N_n}+\log c_{D_n}-\log c_k-\log c_n )\\
&:=I_1(k,n)+I_2(k,n)+I_3(k,n).
\end{align}
When $k=0$, we set $c_k=1$ and set $I_1(0,n)=\frac{1}{L_n}(N_n\log N_n+D_n\log D_n-n\log n)$, $I_2(0,n)=
\frac{1}{2L_n} (\log N_n+\log D_n-\log n)$ and $I_3(0,n)=\frac{1}{L_n}(\log c_{N_n}+\log c_{D_n}-\log c_n )$ to consist with the definitions. The expressions of $I_j$ are different according to the choices of $L_n$ (only differ by the front factor $L_n$), but we use the same notation $I_j$ for different cases throughout the article to reduce the notations we use. 

In the following computations,  we will let $L_n\to\infty$ (although we choose different $L_n$ for different cases), hence $I_3(k,n)$ will tend to $0$ uniformly by the uniform bound of $c_k$, which means the third term $I_3(k,n)$ is always negligible.  
\subsection{Case \textcircled 1}
Let's first consider the case \textcircled 1 when
\begin{equation}\label{case 1}
\lim\limits_{n\to\infty}\frac{N_n}{n} =0.
\end{equation}
For this case, we need to choose $L_n=n$ in \eqref{key}. We first simply have
 \begin{equation}\label{limit of I2}
  \lim\limits_{n\to\infty} \sup_{0\leq k \leq D_n} \abs{I_2(k,n)}\leq  \lim\limits_{n\to\infty} \frac{2}{n}\log n =0.
\end{equation}
For $I_1(k,n)$, we observe that for each fixed $n$, $I_1(k,n)$ is increasing with respect to $k$ by considering the function  $I(x)=(x+N_n )\log(x+N_n )-x\log x$ where $I'(x)=\log (x+N_n )-\log x\geq 0$. We combine this with the fact that $I_1(D_n,n)=\frac{1}{n}((D_n+N_n)\log(D_n+N_n)+D_n\log D_n-n\log n-D_n\log D_n)=0$, we first have
\begin{align*}
 &\sup_{0\leq k \leq D_n}\abs{I_1(k,n)}\leq \abs{I_1(0,n)}\vee \abs{I_1(D_n,n)} =\abs{I_1(0,n)},
  \end{align*}
which further reads 
  \begin{align*}
 \sup_{0\leq k \leq D_n}\abs{I_1(k,n)}&\leq \frac{1}{n}|n\log n-N_n\log N_n-D_n\log  D_n|\\
  &= \frac{1}{n}|N_n\log n+D_n\log n-N_n\log N_n-D_n\log D_n|\\
  &=|-\frac{N_n}{n}\log (\frac{N_n}{n})-\frac{D_n}{n}\log (\frac{D_n}{n})|.
  \end{align*}
Thus, we have
\begin{equation}\label{limit of I1}
  \lim\limits_{n\to\infty} \sup_{0\leq k \leq D_n} \abs{I_1(k,n)}=0,
\end{equation}
since $\frac{N_n}{n}\to 0$ and $\frac{D_n}{n}=1-\frac{N_n}{n}\to 1$ as $n\to \infty$.

Combing \eqref{limit of I2}\eqref{limit of I1} and the fact that $I_3$ always tends to $0$, we get
\begin{equation}\label{i1 in case 1}
\lim\limits_{n\to\infty} \sup_{0\leq k\leq D_n}\abs{\frac{1}{n}\log f_{k,n}}=0.
\end{equation}
Hence, the coefficients $f_{k,n}$ satisfy \textbf{A2}  with $L_n=n$, $T_0=1$, $\delta_n=\frac{N_n}{n}$ so that $(1-\delta_n)L_n=D_n$  and $\log f(t)=0$ for $0\leq t \leq 1$ and
$\log f=-\infty $ for $t>1$. Therefore,   zeros of  $K_{n}^{(N_n)}$ will have the same distribution as the  Kac polynomials by computations in \eqref{examplekac1} and \eqref{examplekac} as $n\to \infty$.

\subsection{Case \textcircled 2}\label{casecase2}
Let's consider the case when
\begin{equation}
\lim\limits_{n\to \infty}\frac{N_n}{n}=a \in (0,1).
\end{equation}
Let's choose $L_n=n$ in \eqref{key} again.  By the same arguments as in Case \textcircled 1,   $I_2$ and $I_3$ converge to $0$ uniformly for $0\leq k\leq D_n$  as
 $n\to \infty$.  Therefore it remains to estimate $I_1$.  Let's put $ \frac{N_n}{n} =a+\delta_n$ where $\delta_n\to0$. Assume $n$ is large enough so that  
\begin{equation}\label{small delta}
\abs{\delta_n} \leq \frac{1-a}{2}\wedge \frac{a}{2}.
\end{equation}
For $k\geq 1$, we rewrite
\begin{align}\label{bound i1 in case 2}
  \nonumber I_1 &=\frac{1}{n} [(k+N_n )\log (k+N_n )+D_n\log D_n-n\log n-k\log k ]\\
 \nonumber &=\frac{1}{n}  [(n-D_n+k)\log(k+N_n ) -n\log n-k\log k+D_n\log k-D_n\log k+D_n\log D_n]\\
  \nonumber &= \frac{1}{n} [n\log (\frac{k}{n}+\frac{N_n}{n})+(k-D_n)\log (k+N_n)-(k-D_n)\log k+D_n\log (\frac{D_n}{k}) ]\\
  \nonumber &= \log (\frac{k}{n}+\frac{N_n}{n})+[(\frac{k}{n}-\frac{D_n}{n})\log (1+\frac{N_n}{k})+\frac{D_n}{n}
  \log (\frac{D_n}{k})] \\
   \nonumber &:= I_4+I_5.
\end{align}
To estimate $I_4$ and $I_5$, we will make use of the following inequality which is the direct consequence of the intermediate value theorem
\begin{equation}\label{control log}
0\leq \log y-\log x\leq \frac{1}{c}(y-x)\quad\mbox{for}\,\,\,0<c\leq x\leq y. 
\end{equation}
We can rewrite $I_4$ as $\log (\frac{k}{n}+a+\delta_n)$, by \eqref{small delta}\eqref{control log},  we have
$$\abs{I_4-\log (\frac{k}{n}+a)}\leq \abs{\frac{2\delta_n}{a}}$$ for all $1\leq k\leq D_n.$
So we have
\begin{equation}\label{boudn i4}
\lim\limits_{n\to \infty} \sup_{1\leq k\leq  D_n}\abs{I_4-\log (\frac{k}{n}+a)}=0.
\end{equation}
For $I_5$, since $\frac{N_n}{n}=a+\delta_n$ and $\frac{D_n}{n}=1-a-\delta_n$, we can rewrite it as
\begin{align*}
  I_5 &= (\frac{k}{n} -(1-a)+\delta_n)\log (1+\frac{(a+\delta_n)n}k)+(1-a-\delta_n)\log \frac{(1-a-\delta_n)n}{k} \\
   &= \frac{k}{n}\log(1+a\frac{n}{k} +\delta_n\frac{n}{k})+(1-a-\delta_n)\log(-1+\frac{n+k}{k+(a+\delta_n)n}).
\end{align*}
Then we have
\begin{align}\label{bound i5}
  \nonumber  &\,\,\,\,\,\, \abs{I_5-[\frac{k}{n}\log(1+a\frac{n}{k})+(1-a)\log(-1+\frac{n+k}{k+an}) ]} \\
  \nonumber  &\leq \frac{k}{n}\abs{\log(1+a\frac{n}{k}+\delta_n \frac{n}{k})-\log(1+a\frac{n}{k})}\\
  \nonumber &+(1-a)
  \abs{\log(-1+\frac{n+k}{k+(a+\delta_n)n})-\log(-1+\frac{n+k}{k+an})}
  \\
  \nonumber &+
  \abs{\delta_n}\abs{\log(-1+\frac{n+k}{k+(a+\delta_n)n})}\\
  \nonumber  &:=I_6+I_7+I_8.
\end{align}
By  \eqref{small delta}\eqref{control log} again,  we have
$$
 {I_6}\leq \frac{k}{n} \frac{1}{1+\frac{a}{2}\frac{n}{k}}\abs{\delta_n}\frac{n}{k}\leq \abs{\delta_n}\to 0.
$$
For $I_7$, since $\abs{\delta_n} \leq \frac{1-a}{2}$, we know $k+(a+\delta_n)n \leq k+\frac{1+a}{2}n$. Therefore, \begin{equation}\label{tititittiti}
	-1+\frac{n+k}{k+(a+\delta_n)n} \geq-1+\frac{n+k}{k+\frac{1+a}{2}n} = \frac{\frac{(1-a)n}{2}}{\frac{1+a}{2}n+k}\geq \frac{\frac{(1-a)n}{2}}{\frac{1+a}{2}n+n }= \frac{1-a}{3+a}.
\end{equation}We also have
$$
-1+\frac{k+n}{k+an}\geq \frac{1-a}{3+a}.
$$
Thus,   by \eqref{control log}, we have 
\begin{align*}
  I_7 &=(1-a)
  \abs{\log(-1+\frac{n+k}{k+(a+\delta_n)n})-\log(-1+\frac{n+k}{k+an})}  \\
   &  \leq (1-a) \frac{3+a}{1-a} \abs{ \frac{k+n}{k+(a+\delta_n)n}-\frac{k+n}{k+an}   }\\
   & \leq (3+a) \frac{(k+n)\abs{\delta_n} n}{(k+\frac{an}{2})^2}\leq (3+a) \frac{(n+n)n\abs{\delta_n}}{(\frac{an}{2})^2} \leq \frac{8(3+a) \abs{\delta_n}}{a^2 }\to 0.
  \end{align*}
For $I_8$, taking into account \eqref{small delta}\eqref{tititittiti}, we have
$$\frac{1-a}{3+a}\leq  -1+\frac{k+n}{k+(a+\delta_n)n}=\frac{(1-a-\delta_n)n}{k+(a+\delta_n)n}\leq \frac{[(1-a)+\frac{1-a}{2}]n}{(a-\frac{a}{2})n}  \leq \frac{3(1-a)}{a}, $$
it follows that
$$
 {I_8}\leq  (\abs{\log(\frac{3(1-a)}{a})}\vee \abs{\log (\frac{1-a}{3+a} )}) \abs{\delta_n}\to 0.
$$
If we combine the estimates of $I_6$, $I_7$ and $I_8$, we conclude that 
\begin{equation}\label{bound i5}
 \lim_{n\to\infty}\sup_{1\leq k\leq D_n} \abs{I_5-[\frac{k}{n}\log(1+a\frac{n}{k})+(1-a)\log(-1+\frac{n+k}{k+an}) ]} = 0
  \end{equation}
If we set
\begin{align}\label{ffffffff}
\nonumber \log f_1(t)&=\log(t+a)+t\log (1+\frac{a}{t})+(1-a)\log(\frac{1-a}{t+a})\\
 &=(t+a)\log (t+a)-t\log t+(1-a)\log (1-a),\,\, t>0,
\end{align}
 then the estimates \eqref{boudn i4}\eqref{bound i5} for $I_4$ and $I_5$ imply
\begin{equation}\label{control i1 in case 2}
  \lim\limits_{n\to\infty}\sup_{1\leq k\leq D_n}\abs{I_1-\log f_1(\frac{k}{n})}=0.
\end{equation}
The estimate of $I_1$ in the case $k=0$ can be achieved by the same way and actually \eqref{control i1 in case 2} holds with the supreme taken over $0\leq k\leq D_n$.

Let's set $f=f_1$ for $0\leq t\leq 1-a$ and $0$ for $t>1-a$. Let's set $$\Delta(b)=\sup_{1-a\leq t\leq s\leq 1, s-t\leq b}\abs{\log f_1(t)-\log f_1(s)},$$ then we have
 $$\sup_{0\leq k\leq D_n}\abs{I_1-\log f(\frac{k}{n}\wedge (1-a))}\leq \sup_{0 \leq k\leq D_n}\abs{I_1-\log f_1(\frac{k}{n})} +\Delta(\abs{\delta_n}).  $$
 Observing that $\log f_1$ is uniformly continuous on $[1-a,1]$,  combining \eqref{control i1 in case 2} and the fact that $\delta_n\to0
 $, then we  have
$$
  \lim\limits_{n\to\infty}\sup_{0\leq k\leq D_n}\abs{I_1-\log f(\frac{k}{n}\wedge(1-a) )}=0.
$$ Therefore, if we combine the estimates of $I_1$, $I_2$ and $I_3$ we derived above for Case \textcircled 2, 
we have
 \begin{equation} \label{control i1 in case 2 final}
   \lim\limits_{n\to\infty}\sup_{0\leq k\leq D_n}\abs{\frac 1n f_{k,n}-\log f(\frac{k}{n}\wedge(1-a) )}=0.
 \end{equation}
 As a summary, in the case when $N_n/n\to a\in (0,1)$, by defining $f(t)$ above, the coefficients $f_{k,n}$ will satisfy \textbf{A2} with $T_0=1-a, L_n=n, \delta_n=\frac{N_n}{n}-a$ (note that $D_n=(T_0-\delta_n)L_n$ again). The  Legendre-Fenchel  transform of $-\log f$ is
\[ I(s)=\begin{cases}
     a\log (\frac{a}{e^{-s}-1}+a)+(1-a)\log (1-a) & s<\log (1-a)  ,   \\  s(1-a)& s\geq \log (1-a).
   \end{cases}
\]
Therefore, by Theorem \ref{theorem2},  the limiting measure  for the sequence of the random measure $\frac{1}{L_n}\mu_{D_n}^K$ (which is $\frac{1}{n}\mu_{D_n}^K$) satisfies
\[ \widehat{\mu}(\D_r)=\begin{cases}
\frac{ar}{1-r}&0<r<1-a, \\ 1-a & r\geq 1-a.
\end{cases}
\]
Since $D_n/n\to 1-a$, thus the limit of the empirical measure $\frac{1}{D_n}\mu_{D_n}^K$ will be $\frac1{1-a}\widehat{\mu}(\D_r) $ which is \eqref{ddddd}.


\subsection{Case \textcircled 3}\label{case333}
In the case when
\begin{equation}\label{case 343}
  \lim\limits_{n\to\infty}\frac{N_n}{n}=1\,\,\,\mbox{and}\,\,\, D_n\to\infty,
\end{equation}
 we only prove \eqref{fff} which   implies \eqref{fdff}. To prove \eqref{fff}, 
we need to consider 
\begin{equation}\tilde{K}_n(z):=R_n^{D_n}K_n^{(N_n)}(\frac{z}{R_n})=\sum_{ k=0}^{D_n}\xi_k\tilde{f}_{k,n}z^k,\end{equation}
where
\begin{equation}\label{def of hat fkn}
\tilde{f}_{k,n}=f_{k,n}R_n^{D_n-k}\,\,\,\mbox{and}\,\,\, R_n =\frac{n}{D_n}.
\end{equation}
It's enough to study $\tilde K_n(z)$ since it has the same zeros as  $K_n^{(N_n)}(\frac{z}{R_n})$.

In this case, we need to choose $L_n=D_n$ in \eqref{key} with the 
decomposition $$\frac{1}{D_n}\log f_{k,n}:=I_1(k,n)+I_2(k,n)+I_3(k,n). $$
Thus we  have the decomposition  \begin{equation}\label{definiii}\begin{aligned}\frac{1}{D_n}\log \tilde{f}_{k,n}&=(1-\frac{k}{D_n})\log R_n+\frac{1}{D_n}\log f_{k,n}\\&=[(1-\frac{k}{D_n})\log R_n+I_1]+I_2+I_3.\end{aligned}\end{equation}
As before, $I_3$ goes to $0$ uniformly again since $D_n\to \infty$ as $n\to \infty$.

 We note that 
 $$I_2(k,n)=\frac{1}{2D_n} (\log (k+N_n)+\log D_n-\log n-\log k)$$
is decreasing with respect to $k\geq 1$ for fixed $N_n$, $D_n$ and $n$, thus we simply have $\sup_{0\leq k\leq D_n}\abs{I_2(k,n)}=\abs{I_2(1,n)}\vee \abs{I_2(D_n,n)}$. Since $I_2(D_n,n)=0$,  we further have
 $$ \sup_{0\leq k\leq D_n} \abs{I_2(k,n)} =\abs{I_2(1,n)}=\frac{1}{2D_n}\abs{\log (N_n+1)+\log D_n-\log n}.$$
By assumption \eqref{case 343}, we can choose $n$ large enough so that $N_n\geq \frac 12 n$, thus we have 
$$ \sup_{0\leq k\leq D_n} \abs{I_2(k,n)}  \leq \frac{1}{2D_n}(\log (\frac{n}{N_n})+\log  D_n)  \leq \frac{\log 2}{2D_n}+\frac{\log D_n}{2D_n}\to 0,$$
since $D_n\to \infty$ as $n\to \infty$. 

For $I_1$, we rewrite it as 
\begin{align*}
I_1&=\frac{1}{D_n} ((k+n-D_n)\log (k+N_n )-n\log n+D_n\log D_n-k\log k)\\&=\frac{1}{D_n}(n\log \frac{k+N_n}{n}+(k-D_n)\log (k+N_n)+D_n\log D_n-k\log k)) \\
 &=\frac{1}{D_n} (n\log (\frac{k+N_n}{n})+(k-D_n)\log (\frac{k+N_n}{n})+(k-D_n)\log n\\&+D_n\log D_n-k\log D_n-k\log(\frac{k}{D_n})  )\\
   &=\frac{n}{D_n}\log (\frac{n+k-D_n}{n})-\frac{k}{D_n}\log (\frac{k}{D_n})\\&+(\frac{k}{D_n}-1)(\log n-\log D_n)+(\frac{k}{D_n}-1)\log (\frac{k+N_n}{n}).
\end{align*}
 Thus we can rewrite 
\begin{align*}
\tilde I_1:&=  (1-\frac{k}{D_n})\log R_n+I_1\\
   &=\frac{n}{D_n}\log (\frac{n+k-D_n}{n})-\frac{k}{D_n}\log (\frac{k}{D_n})+(\frac{k}{D_n}-1)\log (\frac{k+N_n}{n}).
\end{align*}
Now we put
\begin{equation}\label{the log f in case3}
 \log \tilde f=t-1-t\log t \,\,\, \mbox{for} \,\, 0\leq t\leq 1 \,\,\mbox{and} \,\, \log \tilde f= -\infty \,\, \mbox{for}\,\,  t>1.
 \end{equation}
 Then we can write $\tilde{I}_1$ as
\begin{equation}\label{control i1 in case 3}
  \tilde{I}_1=\log \tilde f(\frac{k}{D_n})+I_9,
\end{equation}
where $$I_9=\frac{n}{D_n}[\log (1+\frac{k-D_n}{n})-\frac{k-D_n}{n}]+
(\frac{k}{D_n}-1)\log (\frac{k+N_n}{n}).$$
Since  $|\log (1+x)|\leq \abs{x}$ and $\abs{\log(1+x)-x}\leq x^2$ when $\abs{x}$ is small, then we have the uniform estimate, 
$$\abs{ \log (1+\frac{k-D_n}{n})-\frac{k-D_n}{n} }\leq (\frac{k-D_n}{n})^2\leq (\frac{D_n}{n})^2$$ 
as $n$ large enough, which implies the first term in $I_9$ tends to $0$. 

Note that $1\geq \frac{k+N_n}{n}\geq \frac {N_n} n$, thus $|\log  \frac{k+N_n}{n}|\leq |\log \frac {N_n} n|=\abs{\log (1-\frac{D_n}{n})}\leq \frac{D_n}{n}$, if we combine this with the fact $|\frac{k}{D_n}-1|\leq 1$,   we prove that the second term in $I_9$  also tends to $0$. Hence, $I_9 \to 0$ as $n\to\infty$. Therefore, $$\lim_{n\to\infty} \sup_{0\leq k\leq D_n} |\tilde I_1-\log \tilde f(\frac k{D_n})|= 0.$$
If we combine the estimates of $\tilde I_1$, $I_2$ and $I_3$ above, we have proved  
\begin{equation}\label{limit of hat fkn in case 3}
  \lim\limits_{n\to \infty}\sup_{0\leq k\leq D_n}\abs{ \frac 1{D_n}\log \tilde{f}^{}_{k,n}-\log \tilde f(\frac{k}{D_n}) }=0.
\end{equation}
As a summary,  the coefficients $\tilde f_{k,n}$ satisfy \textbf{A2} with $L_n=D_n$, $T_0=1$, $\delta_n=0$ and $\tilde f$.  The Legendre-Fenchel transform $I(s)=\sup_{0\leq t\leq 1}(st+\log \tilde f(t))$ is
\[ I(s)= \begin{cases}
e^s-1 & s<0, \\s &s\geq 0.
\end{cases}
\] Thus, the explicit expression \eqref{fff}  of the limiting measure $\tilde\mu^K$ follows by Theorem \ref{theorem2}. 
In the end, the existence of the rescaling limit implies that $\frac{1}{D_n}\mu_{D_n}^K$ converges to the Dirac measure centered at $0$. 


\subsection{Case \textcircled 4}\label{tesxt4}
Now we  prove Theorem \ref{theorem4} for the case where $D_n$ remains to be a fixed positive integer $m$. The proof makes use of the Rouch\'e's theorem.
We start with the following proposition regarding the convergence of zeros of a sequence of deterministic polynomials.
\begin{prop}\label{prop 2}
Let $G=\sum_{k=0}^m g_k z^k$, where $\{g_k\}$ are deterministic constants and $g_m\neq 0$. Let $G_n=\sum_{k=0}^m g_{k,n}z^k$, where $\{g_{k,n}\}$ are also deterministic. Assume $g_{k,n}$ converges to $g_k$ for each fixed $k$. Then, the measure of zeros $\mu_{G_n}$ will converge to $\mu_{G}$ in the sense of distribution.
\end{prop}
\begin{proof} 
  Let's choose $\phi$ as the smooth test function with compact support and pick $\epsilon>0$ small enough.
 We first claim that for each zero $z_0$ of $G$ with multiplicity $\alpha_0$,  for $n$ large enough, $G_n$ has exactly $\alpha_0$ zeros  in $\D(z_0,\epsilon)$, the open disc centered at $z_0$ with radius $\epsilon$. Once this is done, since $G$ has $m$ zeros ($m$ is a finite number), we can pick a common $N_0$ such that when $n>N_0$,   $G_n$ will have exactly $\alpha_i$ zeros in $\D(z_i,\epsilon)$ for any $z_i$ in the zero set of $G$ with multiplicity $\alpha_i$. This means that we can make an appropriate ordering of the zero set of $G$ (denoted by $z_i$, $1\leq i\leq m$) and the zero set of $G_n$ (denoted by $z_{i,n}$, $1\leq i\leq m)$ such that $\abs{z_i-z_{i,n}}\leq \epsilon $ for all $i$. Then we have, 
\begin{align}
  \abs{\mu_{G_n}(\phi)-\mu_{G}(\phi)}& \leq \sum_{1\leq i\leq m}\abs{\phi(z_i)-\phi(z_{i,n})} \leq mK\epsilon.
\end{align}
    where $K$ is the sup norm of the derivative of $\phi$. Since $\epsilon$ is arbitrary small, this implies the weak convergence of $\mu_{G_n}$. All the rest is to prove the claim.

 Let's choose $\epsilon <1 $  small enough such that   $z_0$ is the only zero of $G$ with multiplicity $\alpha\geq 1$ in the closure of $\D(z_0,\epsilon)$. Assume $\abs{z_0}+1\leq R$ for some $R$. For any $z \in \overline{\D(z_0,\epsilon )}$, we have
\begin{equation}\label{close gn and g}
  \abs{G_n-G}\leq \sum_{k=0}^m \abs{g_{n,k}-g_k}R^k.
\end{equation}
Let's set $$\eta(\epsilon)=\min_{z\in \partial\D(z_0,\epsilon)}\abs{G(z)},$$
then as $n$ large enough, we have  $$\sum_{k=0}^m \abs{g_{n,k}-g_k}R^k <\eta(\epsilon), $$ which implies that $$ \abs{G_n(z)-G(z)}< \abs{G(z)}\,\,\,\mbox{for any}\,\,\, z\in\partial\D(z_0,\epsilon). $$Hence,  $G_n$ and $G$ has the same number of zeros in $\D(z_0,\epsilon)$ by Rouch\'e's theorem. This completes the proof of claim and hence  Proposition \ref{prop 2}.
\end{proof}
Let's apply Proposition \ref{prop 2} to prove Theorem \ref{theorem4}.  In the case of $D_n=m$ and $N_n=n-m$,  \eqref{def of fn} reads $$K_n^{(n-m)}(z)=\sum_{k=0}^m\xi_kf_{k,n}z^k.$$ 
To study the limiting behavior  of zeros of  $K_n^{(n-m)}(\frac{z}{n})$, we may alternatively consider the random polynomials $G_n(z)=n^mK_n^{(n-m)}(\frac{z}{n})$. The coefficients of $G_n$ is 
\begin{align*}
  g_{k,n} & =n^{m-k}f_{k,n}  =\frac{m!}{k!}\frac{n^{m-k}}{n(n-1)\cdots (n-(m-k)+1)}.
\end{align*}
Since $k$ and $m$ are both fixed when $n\to \infty$, we have
$$\lim\limits_{n\to\infty}g_{k,n}=\frac{m!}{k!} .$$
By Proposition \ref{prop 2}, the measure of zeros   $\mu_{G_n}$ will converge to $\mu_{f^K_m}$ almost surely, where $\mu_{f^K_m}$ is the random  measure of zeros of $f^K_m(z)=\sum_{k=0}^m \frac{\xi_k}{k!}z^k$.
The limit \eqref{ddddffff} follows from this since  $K_n^{(n-m)}(\frac{z}{n})$ have the same  zeros as  $G_n$. 
In particular, the empirical measure of zeros of $K_n^{( n-m)}$ will converges to $\delta_0$.  

\section{General random polynomials}\label{generalpolynomials}
In this section, we will apply the estimates we derived for the Kac polynomials in \S \ref{Kac} to prove
Theorem \ref{theorem5} for the general random polynomials. 

Let $p_n$ be the general random polynomials of degree $n$ defined in \eqref{general}. Let's assume that the coefficients $p_{k,n}$ satisfy \textbf{A1} with the associated continuous function $p$ that is positive on $[0,1)$ and 
\begin{equation}\label{labelde}
	\lim\limits_{n\to\infty}\sup_{k\in[0,n]} \abs{\abs{p_{k,n}}^{\frac{1}{n}}-p(\frac{k}{n})  }=0.
\end{equation}
 The $N_n$-th derivative of $p_n$ is
\begin{equation}\label{def of origin u}
 p_{n}^{(N_n)}=\sum_{k=0}^{D_n} \xi_{k+N_n}  p_{k+N_n,n} f_{k,n}z^k,
\end{equation}
where $f_{k,n}$ is defined in \eqref{def of fk,n}.  
Since $\xi_k$ are i.i.d., it's equivalent to consider the following random polynomials
\begin{equation}\label{def of u}
 p_{n}^{(N_n)}=\sum_{k=0}^{D_n} \xi_kp_{k+N_n,n}f_{k,n}z^k,
\end{equation}
where \eqref{def of origin u} and \eqref{def of u} have the same distribution of zeros. 
We set $$u_{k,n}=p_{k+N_n,n}f_{k,n},$$ then we rewrite
$$ p_{n}^{(N_n)}=\sum_{k=0}^{D_n} \xi_k u_{k,n}z^k.  $$
We now verify that $u_{k,n}$ satisfy    \textbf{A2}  with some associated function $u$. 

\textbf{Case 1}: $N_n/n\to 0$. As in Case \textcircled 1 of Kac polynomials, we take $L_n=n$, $\delta_n=\frac{N_n}{n}$ and $T_0=1$. For fixed $n$, $f_{k,n}$ is increasing with $k$ since 
    $$\frac{f_{k+1,n}}{f_{k,n}}=\frac{k+1+N_n}{k+1}>1.$$ 
    Since $f_{D_n,n}=1$, it follows that
  $f_{k,n}\leq 1$ for all $n$ and $0\leq k\leq D_n$. By the assumptions  \textbf{A1}, $p$ is continuous on $[0,1]$ and therefore is bounded by $C$. Hence, 
\begin{align*}
 &\,\,\, \sup_{0\leq k\leq D_n}  \abs{\abs{u_{k,n}}^{\frac{1}{n}}-p(\frac{k}{n})  } \\&\leq
     \sup_{0\leq k\leq D_n} \abs{\abs{p_{k+N_n,n}}^{\frac{1}{n}}-p(\frac{k}{n})}{|f_{k,n}|}^{\frac{1}{n}}+
    \sup_{0\leq k\leq D_n} \abs{\abs{f_{k,n}}^{\frac{1}{n}}-1}p(\frac{k}{n})\\
    & \leq \sup_{0\leq k\leq D_n}  \abs{\abs{p_{k+N_n,n}}^{\frac{1}{n}}-p(\frac{k+N_n}{n})} + \sup_{0\leq k\leq D_n} \abs{p(\frac{k+N_n}{n})-p(\frac{k}{n})}\\&+
 C\sup_{0\leq k\leq D_n} \abs{|f_{k,n}|^{\frac{1}{n}}-1}\\
 &:=J_1+J_2+J_3
\end{align*}  
Our assumption \eqref{labelde} implies that $J_1$ converges to $0$. 
 $J_2$ converges to $0$  since $p$ is uniformly continuous on $[0,1]$ and $\frac{N_n}{n}$ converges to $0$ under the consideration. $J_3$ also converges to $0$ by the estimate \eqref{i1 in case 1} which we have already proved for the Kac polynomials.
Hence,  the coefficients $u_{k,n}$ satisfy \textbf{A2} with $L_n=n$, $\delta_n=\frac{N_n}{n}$,$T_0=1$ and the associated function $p$. The conclusion (1) of Theorem \ref{theorem5} then follows.

\textbf{Case 2}: $N_n/n\to a \in (0,1)$. As in Case \textcircled 2 of \S \ref{casecase2}, we set $L_n=n$, $\delta_n=\frac{N_n}{n}-a$ and $T_0=1-a$, then $(T_0-\delta_n)L_n=D_n$.  Let's choose $f_1$ as in \eqref{ffffffff} and set that $f$ coincides with $f_1$ in $[0, 1-a]$ and equals to $0$ in $[1-a,\infty)$ as in the Kac case. 
Proceeding like Case 1 above,  we have
\begin{align*}
  &\,\,\,\, \sup_{0\leq k\leq D_n} \abs{\abs{u_{k,n}}^{\frac{1}{n}}-p((\frac{k}{n}+a)\wedge 1)f(\frac{k}{n}\wedge T_0)  }\\
 &\leq \sup_{0\leq k\leq D_n} |f_{k,n}|^{\frac{1}{n}}\abs{\abs{p_{k+N_n,n}}^{\frac{1}{n}}-p((\frac{k}{n}+a)\wedge 1)} \\&+ \sup_{0\leq k\leq D_n} p((\frac{k}{n}+a) \wedge 1)\abs{\abs{f_{k,n}}^{\frac{1}{n}}-f(\frac{k}{n}\wedge T_0)}\\
 &\leq \sup_{0\leq k\leq D_n} |f_{k,n}|^{\frac{1}{n}}\abs{\abs{p_{k+N_n,n}}^{\frac{1}{n}}-p(\frac{n+N_n}{n})}\\&+
 \sup_{0\leq k\leq D_n} |f_{k,n}|^{\frac{1}{n}} \abs{p(\frac{k+N_n}{n})-p((\frac{k}{n}+a) \wedge 1)}\\
 & +
 \sup_{0\leq k\leq D_n} p((\frac{k}{n}+a) \wedge 1)\abs{|f_{k,n}|^{\frac{1}{n}}-f(\frac{k}{n}\wedge T_0)}\\
 &:=J_1+J_2+J_3
\end{align*}
As in Case 1, our assumptions of $p$ imply that $J_1$ converge to $0$; $J_3$ converge to $0$ which is equivalent to \eqref{control i1 in case 2 final} as in the Kac case. Again using the boundedness of $f_{k,n}$ and the uniform continuity of $p$ together with the fact that 
$$\sup_{0\leq k \leq D_n}\abs{((\frac{k}{n}+a) \wedge 1)-\frac{k+N_n}{n}}\leq \abs{\delta_n},$$we have $J_2\to 0$ since $\delta_n\to 0$.  Hence, the coefficients $u_{k,n}$ satisfy \textbf{A2} with $u^a(t)=f(t)p(t+a)$, this will complete the proof of the part (2) in Theorem \ref{theorem5}.

\section{Random elliptic polynomials}\label{ellpticcase}
In this section, we will prove Theorem \ref{theorem6} for  the random elliptic random polynomials $E_n$ defined in
\eqref{elliptic}. Let's denote by
 $p^E_{k,n}=\sqrt{n\choose k}.$ the coefficients.
 By Stirling's formula, one can prove that the coefficients $p_{k,n}^E$
satisfy   \textbf{A1} with the associated function $p^E$ given in \eqref{pe}.    Thus, the part (1) of Theorem \ref{theorem6} is the direct consequence of Theorem \ref{theorem5}.  Now let's prove part (2) of Theorem \ref{theorem6} which is the interesting part and the nontrivial ingredient is to find the rescaling factor. 
As in \eqref{def of u},  the $N_n$-th derivative of $E_n$ is equivalent to
 \begin{equation}\label{altermep}E_{n}^{(N_n)}=\sum_{k=0}^{D_n}\xi_k p^E_{k+N_n,n}f_{k,n}z^k:=\sum_{k=0}^{ D_n}\xi_k u^E_{k,n}z^k.\end{equation}
Let's first consider the case when $N_n/n\to 1$ and $D_n\to\infty$. By discarding a negligible lower order term and by Stirling's formula, we have 
\begin{align}\label{argument}
\nonumber& \frac 1{D_n} \log p^E_{k+N_n,n} \\& \nonumber\sim \frac{1}{2D_n}(n\log n-(k+N_n)\log (k+N_n)-(D_n-k)\log (D_n-k)) \\
\nonumber   & =\frac{1}{2}(\frac{k+N_n}{D_n}\log (\frac{n}{k+N_n})+\frac{D_n-k}{D_n}\log (\frac{n}{D_n-k}))\\
  \nonumber &=\frac{1}{2}( -\frac{n+k-D_n}{D_n}\log (\frac{n-D_n+k}{n})-\frac{D_n-k}{D_n}  \log (\frac{D_n-k}{D_n})+\frac{D_n-k}{D_n} \log (\frac{n}{D_n}))\\
   &=I_{1,1}+I_{1,2}+I_{1,3}.
\end{align}
By $\abs{\log (1+x)-x}\leq x^2$ when $\abs{x}$ is small, we can get the uniform estimate, 
\begin{align*}
  \abs{I_{1,1}-\frac{1}{2}(-\frac{n+k-D_n}{D_n}\frac{-D_n+k}{n}  )} & \leq \frac{n}{2D_n}(\frac{-D_n+k}{n})^2  \leq   \frac{n}{2D_n}(\frac{D_n}{n})^2 \to 0.
\end{align*}  
We also have the uniform estimate
\begin{align*}
  \abs{\frac 12(-\frac{n+k-D_n}{D_n}\frac{-D_n+k}{n}) -\frac{D_n-k}{2D_n}}= \frac{(D_n-k)^2}{2nD_n}\leq \frac{D_n}{2n}\to 0,
\end{align*}
it follows that if we define $$h_1=\frac 12 (1-t),$$
then
\begin{equation}\label{empty}
  \lim\limits_{n\to\infty} \sup_{0\leq k\leq D_n} \abs{I_{1,1}-h_1(\frac k{D_n})}=0.
\end{equation}
Let's put $$h_2=-\frac{1}{2}(1-t)\log (1-t),$$ then we can rewrite 
\begin{equation}\label{i123}I_{1,2}=h_2(\frac{k}{D_n}).\end{equation}
The trick now is to eliminate  $I_{1,3}$ by a rescaling factor. To be more explicit,  let's put $R_n=\frac n{D_n}$ again and put
\begin{equation}\label{scale p}
\tilde{p}^E_{k+N_n,n}=p^E_{k+N_n,n}R_n^{-\frac{D_n-k}2}.
 \end{equation}
 By defining in this way, we note that
 \begin{equation}\label{dsdf}
 	\frac 1{D_n}\log R_n^{-\frac{D_n-k}2}=-I_{1,3},
 \end{equation}
 hence,  if we combine \eqref{argument}-\eqref{dsdf} and define the function \begin{equation}\label{logpe}
 	\log \tilde p^E(x)=h_1+h_2=\frac{1}{2}(1-t)-\frac 12(1-t)\log (1-t),
 \end{equation}then we have proved 
\begin{equation}\label{limit of overline p}
  \lim\limits_{n\to\infty}\sup_{0\leq k\leq D_n}\abs{\frac{1}{D_n}\log \tilde{p}^E_{k+N_n,n}-\log \tilde p^E(\frac{k}{D_n})}=0.
\end{equation}
Let's further recall \eqref{def of hat fkn} in the proof of Case \textcircled 3 for the Kac case where
\begin{equation}\label{scale of f}
  \tilde{f}_{k,n}=f_{k,n}R_n^{D_n-k}, 
\end{equation}
then we can rewrite \eqref{altermep} as
$$E^{(N_n)}_n(z)=\sum_{k=0}^{D_n}\xi_k \tilde{p}^E_{k+N_n,n} \tilde{f}_{k,n}z^k R_n^{-\frac{D_n-k}{2}}.$$ 
Therefore, the rescaling random polynomials reads
\begin{equation}\label{overline u}
 E^{(N_n)}_n(\frac{z}{\sqrt{R_n}})=R_n^{-\frac {D_n}2}\sum_{k=0}^{D_n} \xi_k \tilde{p}^E_{k+N_n,n} \tilde{f}_{k,n}z^k. 
\end{equation}
Let's define
$$ \tilde E^{(N_n)}_n(z):=\sum_{k=0}^{D_n} \xi_k \tilde{p}^E_{k+N_n,n} \tilde{f}_{k,n}z^k.$$
 
Let's derive the limit of the empirical measure of zeros of   $E^{(N_n)}_n(\frac{z}{\sqrt{R_n}})$ which is the same as $ \tilde E^{(N_n)}_n(z)$. To do this, 
let's define the coefficients $\tilde {u}^E_{k,n}:=\tilde {p}^E_{k+N_n,n} \tilde{f}_{k,n}$, then  the estimates \eqref{limit of hat fkn in case 3} and  \eqref{limit of overline p}  imply that  $\tilde{u}^E_{k,n}$  satisfy \textbf{A2} with $L_n=D_n, \delta_n=0, T_0=1$ and  the associated  function  $\tilde {u}^E$ is given by
$\log \tilde{u}^E=\log \tilde p^E+\log \tilde f$. By the expressions \eqref{the log f in case3} and \eqref{logpe}, we have 
\[\log \tilde{u}^E(t)=\begin{cases}
\frac{1}{2}(t-1)-\frac{1}{2}(1-t)\log (1-t)-t\log t & 0\leq t\leq 1,
\\ -\infty &t>1.
\end{cases}
\]
Therefore, $\frac 1{D_n}\mu_{D_n}^{\tilde E}$, or equivalently  $\frac 1{D_n}\mathcal S_{\sqrt{R_n}}(\mu_{D_n}^{ E} )$, converges in probability to a deterministic measure. To find out the limit, we compute the  Legendre-Fenchel  transform of $-\log \tilde{u}^E$ as
\begin{align*}
  I(s) & =\sup_{0\leq t\leq 1}(st+\log \tilde u(t)) =\frac{1}{2}(t_s-1)-\frac 12\log (1-t_s),
\end{align*}
where $t_s=\frac{-1+\sqrt{1+4e^{-2s}}}{2e^{-2s}}$.
Therefore, \eqref{equationdddd} follows by Theorem \ref{theorem2}.

The analysis for the case when $D_n$ remains a fixed number $m$ follows exactly the same approach as in \S \ref{tesxt4} for the Kac case.  Recall the definition of $u^E_{k,n}$ in \eqref{altermep}, if we replace $D_n=m$ and $N_n=n-m$,  then we can rewrite
\begin{align*}
  u^E_{k,n} &=  \left(\frac{n!}{(k+n-m)!(m-k)!}\right)^{\frac{1}{2}}\frac{(k+n-m)!m! }{k!n!}\\
   & = \frac{m!}{k!}\left(\frac{(n-m+k)!  }{n!(m-k)!}\right)^{\frac{1}{2}}.
\end{align*}
Now we consider the rescaling random polynomials $$\tilde E^m_n(z):=n^{\frac{m}{2}}E_n^{(n-m)}(\frac{z}{\sqrt{n}})=\sum_{k=0}^{m}\tilde{u}^E_{k,n}\xi_k z^k,$$
where  $\tilde {u}^E_{k,n}= u^E_{k,n} n^{\frac{m-k}2}$. Since $m$ and $k$ are both fixed when $n\to\infty$, we get
$$\lim\limits_{n\to\infty}\tilde {u}^E_{k,n}=\frac{m!}{k! ((m-k)!)^{\frac{1}{2}} }.$$
Therefore, since $\tilde E^m_n(z)$ have the same zeros as $E_n^{(n-m)}(\frac{z}{\sqrt{n}})$, then by Proposition \ref{prop 2}, the limiting
measure $\mathcal S_{\sqrt n}(\mu_{D_n}^E) $ when $D_n=m$ will tend to the random zeros of 
$$f^E_m=\sum_{k=0}^{m}\frac{1}{k! ((m-k)!)^{\frac{1}{2}}} \xi_kz^k$$ 
in distribution,  which completes the proof of Theorem \ref{theorem6}. 

\appendix
\section{Proof of Theorem \ref{theorem2} }\label{proofoftheorem2}
Now we sketch the proof of Theorem \ref{theorem2} by modifying the one in \cite{KZ}. 

Let's first recall the proof of Theorem \ref{theorem1} in \cite{KZ}. For  random analytic functions $F(z)$ defined in \eqref{analytic} where the coefficients satisfy \textbf{A1}, if one establishes the following convergence in probability
 \begin{equation}\label{proof of thm1}
   \frac{1}{n} \log \abs{F_n(z)}\to I(\log|z|)
 \end{equation}
as $n\to\infty$, then Theorem \ref{theorem1} follows by the classical Poincar\'e-Lelong formula. 
Kabluchko-Zaporozhets proved \eqref{proof of thm1} by establishing some appropriate upper and lower bounds for $\abs{F_n(z)}$, see estimates (22) and (27) in \cite{KZ}. 
 

Under the assumptions \textbf{A2}, the convergence radius is automatic infinity because we are now dealing with a finite sum for any fixed $n$. Given random polynomials $F_n$ in form of \eqref{analytic22} satisfying \textbf{A2}, to prove Theorem \ref{theorem2}, it's enough to derive the analogue convergence   
\begin{equation}\label{proof of thm2}
     \frac{1}{L_n} \log \abs{F_n(z)}\to I(\log|z|)
\end{equation} as $n\to\infty$,
where the convergence is also in probability. 
To prove this, we need the some upper and lower bounds as  in \cite{KZ}. 

For the upper bound, for any $\epsilon>0$, we have
\begin{equation}\label{upper bound}
    \abs{F_n(z)}\leq Me^{L_n(I(\log \abs{z})+3\epsilon+\delta_n^-(\log \abs{z})^+)}  \,\,\,\, \mbox{for  $n$ large enough,}
\end{equation}
where $M$ is an almost surely finite random variable depending on $\epsilon$. Here we use the convention that
for any real number $w$, $w^+$ and $w^-$ are the positive and negative parts of $w$, i.e. $w^+=w\vee 0 $ and $w^-=(-w) \vee 0 $. 

We also need to show the lower bound estimate
\begin{equation}\label{lower bound}
  \PR(\abs{F_n(z)}<e^{L_n(I(\log \abs{z})-4\epsilon)})=O(\frac{1}{\sqrt{L_n}})\quad \mbox{as}\,\,\, n\to\infty.
\end{equation}
Recall Lemma 4.4 in \cite{KZ}, we know that for any $A>0$, there exists an almost surely finite random variable $M'$ such that  $\abs{\xi_k}\leq M' e^{Ak}$ for all $k$ with probability one. If we set $A=\frac{\epsilon}{2T_0}$, then
for all $0\leq k\leq (T_0-\delta_n)L_n$, we have 
\begin{equation}\label{bounding xi}
  \abs{\xi_k}\leq M' e^{\frac{\epsilon k}{2T_0}}\leq M' e^{\epsilon L_n}.
\end{equation}
To prove \eqref{upper bound}, if we apply the bound \eqref{bounding xi} together with the assumptions \textbf{A2}, for $n$ large enough and $\delta$ small enough,  
we have 
\begin{align*}
  \abs{F_n(z)} &=\abs{\sum_{0\leq k\leq  (T_0-\delta_n)L_n}\xi_kp_{k,n}z^k }\leq  \sum_{0\leq k\leq (T_0-\delta_n)L_n}|\xi_k|\abs{p_{k,n}}\abs{z}^k \\
  & \leq M'e^{\epsilon L_n}\left(\sum_{0\leq k \leq (T_0-\delta_n^+)L_n}\abs{p_{k,n}}\abs{z}^k+\sum_{T_0L_n< k\ \leq (T_0+\delta_n^-)L_n }\abs{p_{k,n}}\abs{z}^k \right)\\
  &\leq M'e^{\epsilon L_n}  \sum_{0\leq k \leq (T_0-\delta_n^+)L_n}(e^{\frac{k}{L_n}\log \abs{z}+\log p(\frac{k}{L_n})}+\delta \abs{z}^{\frac{k}{L_n}}  )^{L_n}\\& +M'e^{\epsilon L_n} \sum_{T_0L_n< k \leq (T_0+\delta_n^-)L_n}(e^{(\frac{k}{L_n}-T_0)\log \abs{z}+(T_0\log \abs{z}+\log p(T_0))}+\delta \abs{z}^{\frac{k}{L_n}}  )^{L_n}.
  \end{align*}
  By the definition of the Legendre-Fenchel  transform, we further have
  \begin{align*}
 |F_n(z)|& \leq M'e^{2\epsilon L_n}(e^{I(\log \abs{z})}+\delta (1\vee \abs{z}^{T_0})  )^{L_n}\\&+M'e^{2\epsilon L_n}e^{\delta_n^- (\log \abs{z})^+L_n} (e^{I(\log \abs{z})}+\delta (1\vee \abs{z}^{2T_0}) )^{L_n}\\
  &\leq M''e^{L_n(I(\log \abs{z})+3\epsilon+\delta_n^-(\log \abs{z})^+ ) }.
\end{align*}
where $M''$ is another almost surely finite random variable, which completes the proof of the upper bound. 

For the lower bound \eqref{lower bound}, if we choose the  set $J$ as the one in the proof of (27) in \cite{KZ}, 
then the assumptions $L_n\to\infty$ and $\delta_n\to 0$ imply that the set $\{k:0\leq k\leq (T_0-\delta_n)L_n,\frac{k}{L_n}\in J \}$ has cardinality bounded below by $\frac{\abs{J}}{2}L_n$. The rest proof follows the one in \cite{KZ} by replacing $n$ by $L_n$ and hence the lower bound follows.

\end{document}